\documentclass[12pt, a4paper]{article}

\usepackage{lscape}
\usepackage{amsfonts}
\usepackage{amssymb,amsthm,amsmath,color,mathrsfs}
\usepackage[T1]{fontenc}
\usepackage{lmodern}
\usepackage{hyperref}
\usepackage{tikz-cd}

\newtheorem{Th}{Theorem}[section]

\newtheorem{Lem}[Th]{Lemma}
\newtheorem{Prop}[Th]{Proposition}
\newtheorem{Res}[Th]{Result}

\theoremstyle{definition}
\newtheorem{Def}[Th]{Definition}
\newtheorem{Not}[Th]{Note}
\newtheorem{Rem}[Th]{Remark}
\newtheorem{Ex}[Th]{Example}

\newcommand{\R}{\mathbb R}

\newcommand{\N}{\mathbb N}
\newcommand{\C}{\mathbb C}
\newcommand{\K}{\mathbb K}
\newcommand{\X}{\mathcal X}
\newcommand{\com}[2]{\mathscr C_{#2}(#1)}

\newcommand{\s}{\text{span}}
\newcommand{\cd}[1]{\text{card}(#1)}

\begin{document}

\title{Basis and Dimension of Exponential Vector Space}
\author{Jayeeta Saha$^a$\footnote{Corresponding Author, e-mail : jayeetasaha09@gmail.com
\hspace{.75 cm} sjpm@caluniv.ac.in} and Sandip Jana$^b$\\ \begin{small}$^a$ \emph{Department of Mathematics, Vivekananda College, Thakurpukur, Kolkata-700063, INDIA}
\end{small}\\ \begin{small}$^b$ \emph{Department of Pure Mathematics, University of Calcutta, Kolkata-700019, INDIA }\end{small}\\ 
}
\date{}
\maketitle

\begin{abstract}

Exponential vector space [shortly \emph{evs}] is an algebraic order extension of vector space in the sense that every evs contains a vector space and conversely every vector space can be embedded into such a structure. This evs structure consists of a semigroup structure, a scalar multiplication and a partial order. In this paper we have developed the concepts of basis and dimension of an evs by introducing the ideas of orderly independent set and generating set with the help of partial order and algebraic operations. We have found that like vector space,
 an evs does not contain basis always. We have established a necessary and sufficient condition for an evs to have a basis. It was shown that equality of dimension is an evs property but the converse is not true. We have studied the dimension of subevs and found that every evs contains a subevs with all possible lower dimensions. Lastly we have computed basis and dimension of some evs which help us to explore the theory of basis by creating counter examples in different aspects. 

\end{abstract}



\section{Introduction}

Exponential vector space  is an algebraic ordered extension of vector space. The word `extension' is used because of the fact that every exponential vector space contains a vector space and conversely every vector space can be embedded into such a structure. This structure comprises a semigroup structure, a scalar multiplication and a compatible partial order. We now start with the definition of evs.  

\begin{Def}\cite{evs-quot} Let $(X,\leq)$ be a partially ordered set, `$+$' be a binary operation on 
	$X$ [called \emph{addition}] and `$\cdot$'$:K\times X\longrightarrow X$ be another composition [called \emph{scalar multiplication}, $K$ being a field]. If the operations and the partial order satisfy the following axioms then $(X,+,\cdot,\leq)$ is called an \emph{exponential vector space} (in short \emph{evs}) over $K$ [This structure was initiated with the name \emph{quasi-vector space} or \emph{qvs} by S. Ganguly et al. in \cite{C(X)}].
	\begin{align*}
	A_1&: (X,+)\text{ is a commutative semigroup with identity } \theta\\  
	A_2&: x\leq y\ (x,y\in X)\Rightarrow x+z\leq y+z\text{ and } \alpha\cdot x\leq
	\alpha\cdot y,\  \forall z\in X,  \forall \alpha\in K\\ 
	A_3&:\text{(i)}\ \alpha\cdot(x+y)=\alpha\cdot x+\alpha\cdot y\\ 
	&\quad\text{(ii)}\ \ \alpha\cdot(\beta\cdot x)=(\alpha\beta)\cdot x\\ 
	&\quad\text{(iii)}\  \ (\alpha+\beta)\cdot x \leq \alpha\cdot x+\beta \cdot x\\
	&\quad\text{(iv)}\ \ 1\cdot x=x,\text{ where `1' is the multiplicative identity in }K,\\ 
	& \forall\,x,y\in X,\  \forall\,\alpha,\beta\in K \\
	A_4&: \alpha\cdot x=\theta \text{ iff }\alpha=0\text{ or }x=\theta\\ 
	A_5&: x+(-1)\cdot x=\theta\text{ iff } x\in X_0:=\big\{z\in X:\ y\not\leq z,\,\forall\,y \in X\smallsetminus\{z\}\big \}\\ 
	A_6&: \text{For each }x \in X,\, \exists\,p\in X_0\text{ such that }p\leq x.
	\end{align*}   
	\label{d:evs}\end{Def}

In the above definition, the axiom $ A_3 $ (iii) indicates a rapid growth of the elements of $X$ due to the fact $x+x\geq2x $ and the axiom $ A_6 $ gives some positive sense of each elements. These two facts express the exponential behaviour of the elements of an evs.

In the axiom $ A_5 $,  we can notice that $X_0$ is precisely the set of all minimal elements of the evs $X$ with respect to the partial order on $X$ and forms the maximal vector space (within $X$) over the same field as that of $X$ (\cite{C(X)}). We call this vector space $X_0$ as the `\emph{primitive space}' or `\emph{zero space}' of $X$ and the elements of $X_0$ as `\emph{primitive elements}'. 

Also given any vector space $V$ over some field $K$, an evs $X$ can be constructed (as shown below) such that $V$ is isomorphic to $X_0$. In this sense, ``exponential vector space'' can be considered as an algebraic ordered extension of vector space.

\begin{Ex} \cite{evs-quot}
	Let $X:=\big\{(r,a)\in\R\times V:r\geq0,a\in V\big\}$, where $V$ is a vector space over some field $K$. Define operations and partial order on $X$ as follows : for $(r,a),(s,b)\in X$ and $\alpha\in K$,\\
	(i) $(r,a)+(s,b):=(r+s,a+b)$;\\ 
	(ii) $\alpha(r,a):=(r,\alpha a)$, if $\alpha\neq0$ and $0(r,a):=(0,\theta)$, $\theta$ being the identity in $V$;\\ 
	(iii) $(r,a)\leq(s,b)$ iff $r\leq s$ and $a=b$.\\ Then $X$ becomes an exponential vector space over $K$ with the primitive space $\{0\}\times V$ which is evidently isomorphic to $V$.
\end{Ex}

Initially the idea of this structure was given by S. Ganguly et al. with the name ``quasi-vector space'' in \cite{C(X)} and the following example of the hyperspace was the main motivation behind this new stucture.

\begin{Ex}\cite{C(X)}
	Let $\com{\X}{}$ be the topological hyperspace consisting of all non-empty compact subsets of a Hausd\"{o}rff topological vector space $\X$ over the field $ \K $ of real or complex numbers. Then $ \com{\X}{} $ becomes an evs with respect to the operations and partial order defined as follows. For $A,B\in\com{\X}{}$ and $\alpha\in\K$,\\
	(i) $A+B:=\{a+b:a\in A,b\in B\}$\\ 
	(ii) $\alpha A:=\{\alpha a:a\in A\}$\\
	(iii) The usual set-inclusion as the partial order.
\label{CX}\end{Ex}

We now topologise an exponential vector space. For this we need the following concept.

\begin{Def}\cite{Nach} Let `$\leq$' be a preorder in a topological
space $Z$; the preorder is said to be \emph{closed} if its graph
$G_{\leq}(Z):=\big\{(x,y)\in Z\times Z:x\leq y\big\}$ is closed in $Z\times Z$ (endowed
with the product topology).
\end{Def}

\begin{Th} {\em\cite{Nach}} A partial order `$\leq$' in a topological space $Z$ will be a closed order iff for any $x,y\in Z$ with
$x\not\leq y$, $\exists$ open neighbourhoods $U,V$ of $x,y$ respectively in
$Z$ such that $(\uparrow U)\cap(\downarrow V)=\emptyset$, where
$\uparrow U:=\{z\in Z:z\geq u \text{ for some }u\in U\}$ and
$\downarrow V:=\{z\in Z:z\leq v \text{ for some }v\in V\}$.
\label{t:clord}\end{Th}

\begin{Def}\cite{evs-quot}  An exponential vector space $X$ over the field $\K$ of real or complex numbers is said to be a \emph{topological exponential vector space} if there exists a topology on $X$ with respect to which the addition and the  scalar multiplication are continuous and the partial order `$\leq$' is
closed (Here $\K$ is equipped with the usual topology).
\end{Def}

\begin{Rem} If $X$ is a topological exponential vector space then its primitive space $X_0$ becomes a topological vector space, since restriction of a continuous function is continuous. Moreover, the closedness of the partial order `$\leq$' in a topological exponential vector space $X$ readily implies (in view of Theorem \ref{t:clord}) that $X$ is Hausd\"{o}rff and hence $X_0$ becomes a Hausd\"{o}rff topological vector space.
\end{Rem}

\begin{Ex}\cite{mor} Let $X:=[0,\infty)\times V$, where $V$ is a vector space over the field $\K$ of real or complex numbers. Define operations and partial order on $X$ as follows : for $(r,a),(s,b)\in X$ and $\alpha\in\K$,\\
(i) $(r,a)+(s,b):=(r+s,a+b)$,\\ 
(ii) $\alpha(r,a):=(|\alpha|r,\alpha a)$,\\ 
(iii) $(r,a)\leq(s,b)$ iff $r\leq s$ and $a=b$.\\
Then $[0,\infty)\times V$ becomes an exponential vector space with the primitive space $\{0\}\times V$ which is clearly isomorphic to $V$.

In this example, if we consider $V$ as a Hausd\"{o}rff topological vector space then $[0,\infty)\times V$ becomes a topological exponential vector space with respect to the product topology, where $[0,\infty)$ is equipped with the subspace topology inherited from the real line $\R$.

If instead of $V$ we take the trivial vector space $\{\theta\}$ in this example then, the resulting topological evs is $[0,\infty)\times\{\theta\}$ which can be clearly identified with the half ray $[0,\infty)$ of the real line.
\label{e:0inf}\end{Ex}

In this paper we have developed the concept of basis and dimension of evs. We know that basis of a vector space is a minimal part of it which generates the entire space. But in evs it is impossible to express every element as a linear combination of some particular elements due to the exponential behaviour of its elements. In this paper, with the help of partial order we have developed the ideas of generating sets, orderly independent sets which help us to define basis. It has been shown that basis of an evs is identified by a minimal generating set whereas maximal orderly independent set fails to form a basis [shown by a counter example], though every basis is a maximal orderly independent set. The main difference between a vector space and an evs in this respect is that an evs may not have a basis always (like vector space). But for a topological evs, we have shown that, if it has a basis then it contains uncountably many bases. We have found out a property of every element of a basis which helped us to give a necessary and sufficient condition for an evs to have a basis. After that we have introduced the concept of dimension of an evs and shown that equality of dimension is an evs property, though two non order-isomorphic evs may have same dimension. 

Lastly we have studied the dimension of subevs and shown that every evs contains subevs(s) with all possible lower dimensions. In the last section of this paper computation of basis and dimension of some evs are shown.  
 
\section{Prerequisites}

In this section we have discussed some definitions, results and examples of exponential vector space which are very much required to develop the main context. We now first start with the definition of subevs. 

\begin{Def} \cite{spri} A subset $Y$ of an exponential vector space $X$ is said to be a \emph{sub exponential vector space} 
(\emph{subevs} in short) if $Y$ itself is an exponential vector space with respect to the compositions of 
$X$ being restricted to $Y$.
\label{d:subevs}\end{Def} 

\begin{Not} \cite{spri} A subset $Y$ of an exponential vector space $X$ over a field $K$ is a sub exponential vector space iff $Y$ satisfies the following :\\ 
(i) $\alpha x+y\in Y,\ \forall\,\alpha\in K$, $\forall\,x,y \in Y$.\\ 
(ii) $Y_{0}\subseteq X_{0}\bigcap Y$, where $Y_0:=\big\{z\in Y:y\nleq z,\forall\,y\in Y\smallsetminus\{z\}\big\}$\\ 
(iii) For any $y\in Y$, $\exists\,p\in Y_{0}$ such that $p\leq y$. 

If $Y$ is a subevs of $X$ then actually $Y_0=X_0\cap Y$, since for any $Y\subseteq X$
we have $X_0\cap Y\subseteq Y_0$. $[0,\infty)\times\{\theta\}$ is clearly a subevs of the evs $[0,\infty)\times V$.
\label{n:subevs}\end{Not}

We have used the following result to form a non-topological exponential vector space.

\begin{Res} {\em\cite{qvs}} In a topological evs $X$ if $a=a+x$ for some $a,x\in X$ then $x=\theta$.
	\label{ch1:r:idem}\end{Res}
	
	To talk about an evs property of this space we have to know the idea of order-morphism.
	
\begin{Def}\cite{mor} A mapping $f:X\longrightarrow Y$ ($X,Y$ being two exponential vector spaces over the field $K$) is called an \emph{order-morphism} if\\
(i) $f(x+y)=f(x)+f(y)$, $\forall\,x,y\in X$\\
(ii) $f(\alpha x)=\alpha f(x)$, $\forall\,\alpha\in K$, $\forall\,x\in X$\\
(iii) $x\leq y$ $(x,y\in X)\Rightarrow f(x)\leq f(y)$\\
(iv) $p\leq q$ $\big(p,q\in f(X)\big)\Rightarrow f^{-1}(p)\subseteq\downarrow f^{-1}(q)$
and $f^{-1}(q)\subseteq\uparrow f^{-1}(p)$.

A bijective (injective, surjective) order-morphism is called an \emph{order-isomorphism} (\emph{order-monomorphism}, \emph{order-epimorphism} respectively).

If $X,Y$ are two topological evs over $\K$ then an order-isomorphism $f:X\longrightarrow Y$ is said to be a \emph{topological order-isomorphism} if $f$ is a homeomorphism.
\label{d:mor}\end{Def}

\begin{Def} A property of an evs is called an \emph{evs property} if it remains invariant under order-isomorphism.
\end{Def}

The concept of order-isomorphism is competent enough to extract the structural beauty of an evs by judging the invariance of its various properties. Since the composition of two order-isomorphisms, the inverse of an order-isomorphism and the identity map are again order-isomorphisms, the concept thereby produces a partition on the collection of all evs over some common field; this helps one to distinguish two evs belonging to two different classes under this partition.

\begin{Def}\cite{spri} In an evs $X$ the \emph{primitive} of $x\in X$ is defined as the set\\ \centerline{$P_x:=\{p\in X_\circ : p\leq x\}$} The axiom $A_6$ in Definition \ref{d:evs} ensures that the primitive of each element of an evs is nonempty.
\end{Def}

\begin{Def}\cite{spri} An evs $X$ is said to be a \emph{single primitive evs} if $P_x$ is a singleton set for each $x\in X.$ Also, in a single primitive evs $X$, $P_{x+y}=P_x+P_y$ and $P_{\alpha x}=\alpha P_x$, $\forall\,x, y\in X$ and for all scalar $\alpha$.

Single primitivity is an evs property \cite{spri}.
\end{Def}

\begin{Def} \cite{spri} An evs $X$ is said to be a \emph{comparable evs} if $\forall x,y\in X$, $P_x=P_y\Rightarrow$ $x$ and $y$ are comparable with respect to the partial order of $X$.
	
	This is also an evs property \cite{spri}.
\end{Def}

We now give some examples of exponential vector space to build up some counter examples of the main section. 

\begin{Ex}\cite{qvs}\label{e:prod}
	{\bf(Arbitrary product of exponential vector spaces)} Let $\{X_i:i\in\Lambda\}$ be an arbitrary family of exponential vector spaces over a common field $K$ and $X:=\displaystyle\prod_{i\in\Lambda}X_i$ be the Cartesian product. Then, $X$ becomes an exponential vector space over $K$ with respect to the following operations and partial order :
	
	For $x=(x_i)_i, y=(y_i)_i\in X$ and $\alpha\in K$ we define (i) $x+y:=(x_i+y_i)_i$, (ii) $\alpha x:=(\alpha x_i)_i$, (iii) $x\ll y$ if $x_i\leq y_i$, $\forall\,i\in\Lambda$.\\
	Here the notation $x=(x_i)_i\in X$ means that the point $x\in X$ is the map $x:i\mapsto x_i\, (i\in\Lambda)$, where $x_i\in X_i,\ \forall\,i\in\Lambda$. The additive identity of $X$ is given by $\theta=(\theta_i)_i$, $\theta_i$ being the additive identity in $X_i$. Also the primitive space of $X$ is given by $X_0=\displaystyle\prod_{i\in\Lambda}[X_i]_0$.
	
	This product space $X$ becomes a topological exponential vector space over the field $\K$ whenever each factor space $X_i$ is a topological evs over $\K$ and $X$ is endowed with the product topology, which is the weakest topology on $X$ so that each projection map $p_i:X\longrightarrow X_i$ given by $p_i:x\longmapsto x_i$ is continuous.
	
	Thus for any cardinal number $\beta$, $[0,\infty)^{\beta}$ becomes a topological evs.
\end{Ex}

\begin{Ex}
	Let $X$ be an evs over the field $\K$ (either $\R$ or $\C$) and $V$ be a vector space
	over the same field $\K$. We now give operations on $X\times V$ like $[0,\infty)\times V$, i.e.\\
	for $(x_1,e_1),(x_2,e_2),(x,e)\in X\times V$ and $\alpha\in\K$\\
	(i) $(x_1,e_1)+(x_2,e_2):=(x_1+x_2,e_1+e_2)$.\\
	(ii) $\alpha(x,e):=(\alpha x,\alpha e)$.\\ The partial order `$\leq$' is defined as : $(x_1,e_1)\leq(x_2,e_2)$ iff
	$x_1\leq x_2$ and $e_1=e_2$. Then $X\times V$ becomes an evs over the field $\K$. Justification of this is straight forward.\label{ch3:e:xxe}\end{Ex}

\begin{Ex} Let $\com{\X}{\theta}$ be the collection of all compact subsets of a Hausd\"{o}rff topological
	vector space $\X$ containing $\theta$ (the identity in $\X$). So $\com{\X}{\theta}\subseteq \com{\X}{}$.
	If we take any two members $A,B\in\com{\X}{\theta}$ and any $\alpha\in \K$ then
	$\alpha A+ B\in \com{\X}{\theta}$. Again $[\com{\X}{\theta}]_{0}=\big\{\{\theta\}\big\}$=
	$[\com{\X}{}]_0\cap \com{\X}{\theta}$. For any $A\in \com{\X}{\theta}$,
	$\{\theta\}\subseteq A$. This shows that $\com{\X}{\theta}$ is a subevs of
	$\com{\X}{}$ [by note \ref{n:subevs}].
	\label{e:topcthetax}\end{Ex}

\begin{Ex}\cite{spri}\label{ch3:e:LX} Let $\X$ be a vector space over the field $\K$ of real or complex numbers. Let
	$\mathscr L(\X)$ be the set of all linear subspaces of $\X$. We now define $+,\cdot,\leq$ on $\mathscr L(\X)$ as follows :
	For $\X_1,\X_2\in\mathscr L(\X)$ and $\alpha\in\K$ define\\
	(i) $\X_1+\X_2:=\s(\X_1\cup \X_2)$, (ii) $\alpha\cdot \X_1:=\X_1$, if $\alpha\neq0$
	and $\alpha\cdot \X_1:=\{\theta\}$, if $\alpha =0$
	($\theta$ being the additive identity of $\X$), (iii) $\X_1\leq \X_2$ iff 
	$\X_1\subseteq \X_2$.
	
	Then $\big(\mathscr L(\X),+,\cdot,\leq\big)$ is an exponential vector space over $\K$.
	
	Since every element of $\mathscr L(\X)$ is an idempotent 
	\big[$\because \X_1+\X_1=\X_1,$ for all $\X_1\in \mathscr L(\X)$\big] we 
	can say that there is no topology with respect to which $\mathscr L(\X)$ can
	be a topological evs \big[Since a topological evs cannot contain any
	idempotent element, as follows from the Result \ref{ch1:r:idem}\big].\qed
\end{Ex}

\begin{Ex}\cite{bal}\label{ch4:e:d2inf}
	Let us consider $\mathscr D^2([0,\infty)):=[0,\infty)\times [0,\infty)$.
	We define $+,\cdot,\leq$ on $\mathscr D^2([0,\infty))$ as follows :\\
	For $(x_1,y_1),(x_2,y_2)\in\mathscr D^2([0,\infty))$ and $\alpha\in\C$ we define\\
	(i) $(x_1,y_1)+(x_2,y_2)=(x_1+x_2,y_1+y_2)$\\
	(ii) $\alpha\cdot(x_1,y_1)=(|\alpha|x_1,|\alpha|y_1)$\\
	(iii) $(x_1,y_1)\leq (x_2,y_2)\Longleftrightarrow$ either $x_1<x_2$ or if
	$x_1=x_2$ then $y_1\leq y_2$ [\emph{dictionary order}]\\
	Then $\big(\mathscr D^2([0,\infty)),+,\cdot,\leq\big)$ becomes a non-topological exponential vector space over the complex field $\C$.
\end{Ex}

\begin{Not}\cite{bal}\label{ch4:n:dxi} For a \emph{\textbf{well-ordered}} set $I$  and an evs $X$, if we consider $\mathscr D(X:I):=X^I$ then also,
	like above example, it forms a non-topological evs with dictionary order. If $I=\{1,2,\dots,n\}$
	we shall usually denote the evs $\mathscr D(X:I)$ as $\mathscr D^n(X)$.  We can also generalise
	this by taking different evs i.e. $\mathscr D(X_\alpha:\alpha \in I):=\displaystyle \prod_{\alpha\in I} X_\alpha$, which
	also becomes a non-topological evs with dictionary order.\end{Not}

\section{Basis and dimension : General discussion}

In this section we have introduced the concepts of basis and dimension of an exponential vector space. These concepts are different from those already in a vector space. Like vector
space it is not true that every evs contains a basis, rather it behaves like a
module in this respect. We have found a necessary as well as sufficient condition
for an evs to have a basis. Finally, we have computed basis and dimension of 
some particular evs.

\begin{Def}Let $X$ be an evs over the field $K$ and $x\in X\smallsetminus X_0$. Define\\
	\centerline{$L(x):=\{z\in X: z\geq \alpha x+p, \alpha \in K^*, p\in X_0\},\ 
		\text{ where } K^*\equiv K\smallsetminus \{0\}$} We name these sets $L(x)$ for
	different $x$'s in $X\smallsetminus X_0$ as \emph{testing sets} of $X$.
	\label{ch5:d:lx}\end{Def}

We discuss below some properties of $L(x)$. First of all note 
that $L(x)=\uparrow( K^*x+X_0)$.

\begin{Prop} {\em(i)} $\forall\,x\in X\smallsetminus X_0$, $x\in L(x)$ and 
	$\uparrow L(x)=L(x)$.\\
	{\em(ii)} $x\leq y\ (x,y\in X\smallsetminus X_0)$ $\Rightarrow$ $L(x)\supseteq L(y)$.\\
	{\em(iii)} If  $x=\alpha y+p$ for some $\alpha \in  K^*$, $p\in X_0$ and
	$y\in X\smallsetminus X_0$, then $L(x)=L(y)$.\\
	{\em(iv)} $L(x)\cap X_0=\emptyset$.\\
	{\em(v)} If $a\in L(b)$ then $L(a)\subseteq L(b)$.\\
	{\em(vi)} For any $x,y\in X\smallsetminus X_0$, $L(x)\cap L(y)\neq \emptyset$.
	\label{ch5:p:L(x)}\end{Prop}

\begin{proof}\textbf{(i)} Immediate from definition.
	
	\textbf{(ii)} Let $z\in L(y)$ $\Rightarrow$ $\exists$ $\alpha \in K^*$ and $p\in X_0$
	such that $\alpha y+p\leq z$. Now $x\leq y$ $\Rightarrow$ $\alpha x+p\leq \alpha y+p\leq z$ $\Rightarrow$ $z\in L(x)$.
	
	\textbf{(iii)} As $y\in X\smallsetminus X_0$, so $x\in X\smallsetminus X_0$.
	Therefore we can talk about $L(x)$.
	Let $z\in L(y)$ $\Rightarrow$ $\exists$ $\alpha_z\in  K^*$ and $p_z\in X_0$ such that
	$\alpha_zy+p_z\leq z$ $\Rightarrow$ $\alpha_z\alpha^{-1}(x-p)+p_z\leq z$ $\Rightarrow$
	$\alpha_z\alpha^{-1}x+(p_z-\alpha_z\alpha^{-1}p)\leq z$ $\Rightarrow$ $z\in L(x)$.
	Therefore $L(y)\subseteq L(x)$. Again $x=\alpha y+p$ $\Rightarrow$ $y=\alpha^{-1}(x-p)$.
	So by above argument we also have $L(x)\subseteq L(y)$. Thus $L(x)=L(y)$.
	
	\textbf{(iv)}  Let $y\in L(x)$ $\Rightarrow$ $\exists$ $\alpha \in  K^*$ and $p\in X_0$ such that $\alpha x+p\leq y$. 
	If $y\in X_0$ then
	$\alpha x+p\in X_0$ $\Rightarrow$ $x\in X_0$. This contradiction proves that $L(x)\cap X_0=\emptyset$.
	
	\textbf{(v)} $a\in L(b)$ $\Rightarrow$ $a\in X\smallsetminus X_0$ and $\exists\,\alpha\in  K^*$, $p\in X_0$
	such that $\alpha b+p\leq a$ $\Rightarrow$ $L(a)\subseteq L(\alpha b+p)=L(b)$
	[by (ii) and (iii) above].
	
	\textbf{(vi)} For any $p\in X_0$ with $p\leq y$, $x+p\leq x+y$ $\Rightarrow$ $x+y\in L(x)$.
	Similarly we can say that $x+y\in L(y)$. So, $x+y\in L(x)\cap L(y)$ $\Rightarrow$ $L(x)\cap L(y)\neq\emptyset$.
\end{proof}

\begin{Def}A subset $B$ of $X\smallsetminus X_0$ is said to be a 
	\emph{generator} of $X\smallsetminus X_0$
	if\\ \centerline{$X\smallsetminus X_0=\displaystyle \bigcup_{x\in B} L(x)$}\end{Def}

\begin{Not} The set $X\smallsetminus X_0$ always generates $X\smallsetminus X_0$. So generator always exists for $X\smallsetminus X_0$.
	It is clear that any superset of a generator of $X\smallsetminus X_0$ is also a generator of $X\smallsetminus X_0$.
\end{Not}

\begin{Def} Two elements $x,y\in X\smallsetminus X_0$ are said to be 
	\emph{orderly dependent}
	if either $x\in L(y)$ or $y\in L(x)$. \end{Def}

\begin{Def} Two elements $x,y\in X\smallsetminus X_0$ are said to be
	\emph{orderly independent}
	if they are not orderly dependent i.e neither $x\in L(y)$ nor $y\in L(x)$.
	
	A subset $B$ of $X\smallsetminus X_0$ is said to be \emph{orderly independent} if
	any two members $x,y \in B$ are orderly independent.\end{Def}

\begin{Rem} Let $Y$ be a subevs of an evs $X$. Then any two orderly dependent elements of
	$Y\smallsetminus Y_0$ are also orderly dependent in $X\smallsetminus X_0$ because of the
	fact $Y_0\subseteq X_0$. In other words, any two elements of $Y\smallsetminus Y_0$
	which are orderly independent in $X\smallsetminus X_0$ are
	also orderly independent in $Y\smallsetminus Y_0$. But converse is not true in general i.e two orderly
	independent elements in $Y\smallsetminus Y_0$ may not be orderly independent in $X\smallsetminus X_0$
	[in contrast to the case of linear independence in vector space].
	For example, $\{0, 2,5\}$ and $\{0,-2,3\}$ are orderly independent in
	$\com{\R}{\theta}$ [see Example \ref{e:topcthetax}], since $\nexists$ any $\alpha \in \R^*$ such that $\alpha \{0,2,5\}\subseteq \{0,-2,3\}$
	or $\alpha \{0,-2,3\}\subseteq \{0,2,5\}$ \big[Here $[\com{\R}{\theta}]_0=\big\{\{0\}\big\}$ \big].
	But these two elements are not orderly independent in $\com{\R}{}$, as
	we can write $\{0,-2,3\}=\{0,2,5\}+\{-2\}$, where $\{-2\}\in [\com{\R}{}]_0$.
	
	In the above context it should thus be noted that while discussing the orderly
	independence of two elements of a subevs $Y$ of an evs $X$, there are two types
	of orderly independence ------ one with respect to $Y$ and the other with respect
	to $X$; while considering orderly independence with respect to $Y$ the testing sets
	should be of the form\\ \centerline{$L_Y(y):=\{z\in Y: z\geq \alpha y+p, \alpha \in  K^*, p\in Y_0\}$ for 
		any $y\in Y\smallsetminus Y_0$,} and when considering orderly independence with
	respect to $X$ the testing sets must be of the form\\
	\centerline{$L_X(y):=\{z\in X: z\geq \alpha y+p, \alpha \in  K^*, p\in X_0\}$ for 
		any $y\in Y\smallsetminus Y_0$.} Since $Y_0\subseteq X_0$ we have $L_Y(y)\subseteq L_X(y)$,
	for any $y\in Y\smallsetminus Y_0$. Thus it follows that an orderly independent
	set in $Y\smallsetminus Y_0$ need not be orderly independent in $X\smallsetminus X_0$.
	However a set $B\,(\subseteq Y\smallsetminus Y_0)$ which is orderly independent in
	$X\smallsetminus X_0$ must be so in $Y\smallsetminus Y_0$.
	\label{ch5:rm:suboi}\end{Rem}

\begin{Def} A subset $B$ of $X\smallsetminus X_0$ is said to be a \emph{basis} of $X\smallsetminus X_0$
	if $B$ is orderly independent and  generates $X\smallsetminus X_0$.\end{Def}

\begin{Not} For each $x\in X$ either $x\in X_0$ or $x\in X\smallsetminus X_0$. If
	$x\in X_0$ then it can be expressed as a finite linear combination of some
	basic vectors of some basis of $X_0$ [as a vector space]. If $x\in X\smallsetminus X_0$
	then there exists a member of some basis [if exists] of $X\smallsetminus X_0$
	which generates $x$. So we can say that
	to represent an evs $X$ it is necessary to consider a basis of $X\smallsetminus X_0$
	together with a basis of $X_0$ [in the sense of vector space]. Thus a basis of
	an evs $X$ should be composed of two components, one for $X_0$ and the other for 
	$X\smallsetminus X_0$. To express this fact in an easiest way we shall represent a
	basis of an evs $X$ as $[B:B_0]$, where $B$ is a basis of $X\smallsetminus X_0$
	and $B_0$ is a basis of $X_0$ [as a vector space]. If for an evs $X$, $X_0=\{\theta\}$
	then we shall consider $B_0=\{\theta\}$, since in that case $X_0$ has no basis.
\end{Not}

\begin{Th} For a topological evs $X$, $X\smallsetminus X_0$ either has no 
	basis or has uncountably many bases.\end{Th}

\begin{proof}Let $B$ be a basis of $X\smallsetminus X_0$. Then $G_\alpha:=\{\alpha x:x\in B\}$
	and $H_p:=\{x+p:x\in B\}$ are also bases of $X\smallsetminus X_0$, for any
	$\alpha \in \K^*$ and any $p\in X_0$. This holds because of the result
	$L(x)=L(\alpha x+p)$ [proposition \ref{ch5:p:L(x)}]. If $G_\alpha=G_\beta$ for 
	any $\alpha,\beta\in\K^*$ then $\alpha x=\beta x$, $\forall\,x\in B$ 
	[$\because$ $\alpha x\neq\beta y$ for any $x,y\in B$ as $B$ is orderly independent].
	If we choose $\alpha,\beta\in\K^*$ such that $|\alpha|<|\beta|$ then using 
	continuity of the scalar multiplication of the topological evs $X$ we must have
	$x=\theta$ [$\because$ $\alpha x=\beta x$ $\Rightarrow$ 
	$(\alpha\beta^{-1})^nx=x$, $\forall\,n\in\N$ which implies by taking limit 
	$n\to\infty$ that $x=\theta$, as $|\alpha\beta^{-1}|<1$] --- a contradiction.
	Thus it follows that for any $\alpha,\beta\in\K^*$ with $|\alpha|<|\beta|$ we
	must have $G_\alpha\neq G_\beta$. This immediately justifies that there are 
	uncountably many bases of $X\smallsetminus X_0$.
	
	If $X_0$ contains more than one element then for $p,q\in X_0$
	we may consider $H_p,H_q$. If $H_p=H_q$ then $B$ being orderly independent we
	must have $x+p=x+q$, $\forall\,x\in B$. Then by result \ref{ch1:r:idem} it follows 
	that $p=q$. Since $X$ is a topological evs, $X_0$ is a Hausd\"{o}rff topological
	vector space. So if $X_0\neq\{\theta\}$ then it must be uncountable and hence
	ensures the existence of uncountably many bases of $X\smallsetminus X_0$. 
\end{proof}

For a non-topological evs it may so happen that $G_\alpha=B$ for every
$\alpha \in K^*$ [this will be discussed in the next section]. However, an evs (topological or not) need not have a basis. We show in the next 
section that the evs $\mathscr D\big([0,\infty):\N\big)$ discussed in Note \ref{ch4:n:dxi} cannot have a basis. The following result shows that having basis is an evs property.

\begin{Res} Let $\phi:X\longrightarrow Y$ be an order-isomorphism. Then\\
	\emph{(1)} for any generator $B$ of $X\smallsetminus X_0$,
	$\phi(B)$ is a generator of $Y\smallsetminus Y_0$.\\
	\emph{(2)} for any orderly independent subset $B$ of $X\smallsetminus X_0$, 
	$\phi(B)$ is also an orderly independent subset of $Y\smallsetminus Y_0$.
	
	Thus, for a basis $B$ of $X\smallsetminus X_0$, $\phi(B)$ becomes a basis of 
	$Y\smallsetminus Y_0$. \label{ch5:r:basisiso}\end{Res}

\begin{proof}
	(1) $B\subseteq X\smallsetminus X_0$ $\Rightarrow$ $\phi(B)\subseteq Y\smallsetminus Y_0$
	[As $\phi(X_0)=Y_0$]. Let $y\in Y\smallsetminus Y_0$ $\Rightarrow$
	$\phi^{-1}(y)\in X\smallsetminus X_0$ $\Rightarrow$ $\exists$ $b\in B$ and
	$\alpha \in  K^*$, $p\in X_0$
	such that $\phi^{-1}(y)\geq \alpha b+p$ $\Rightarrow$ $y\geq \alpha \phi(b)+\phi(p)$
	$\Rightarrow$ $y\in L(\phi(b))\subseteq \displaystyle \bigcup_{b\in B}L(\phi(b))$.
	Therefore $Y\smallsetminus Y_0\subseteq \displaystyle \bigcup_{b\in B}L(\phi(b))$.
	Again by the proposition \ref{ch5:p:L(x)}, $L(\phi(b))\cap Y_0=\emptyset$, $\forall\,b\in B$.
	So $Y\smallsetminus Y_0=\displaystyle \bigcup_{b\in B}L(\phi(b))$. Thus 
	$\phi(B)$ is a generator of $Y\smallsetminus Y_0$.
	
	(2) We first show that for any two orderly dependent members $y_1,y_2 $ of
	$Y\smallsetminus Y_0$, $\phi^{-1}(y_1),\phi^{-1}(y_2)$ are orderly dependent
	in $X\smallsetminus X_0$. As $y_1,y_2$ are orderly dependent so without
	loss of generality we can take $y_1\in L(y_2)$ $\Rightarrow$ $\exists$ $\alpha \in  K^*$ and
	$p\in Y_0$ such that $\alpha y_2+p\leq y_1$. Then  $\phi^{-1}$ also being an 
	order-isomorphism we have 
	$\phi^{-1}(\alpha y_2+p)\leq \phi^{-1}(y_1)$ $\Rightarrow$
	$\alpha \phi^{-1}(y_2)+\phi^{-1}(p)\leq \phi^{-1}(y_1)$ $\Rightarrow$ 
	$\phi^{-1}(y_1)\in L(\phi^{-1}(y_2))$ [as $\phi^{-1}(p) \in X_0$].
	This justifies our assertion. Then contra-positively, the result follows.\end{proof}

The next theorem characterises a basis (if exists) of $X\smallsetminus X_0$, 
for any evs $X$.

\begin{Th} A subset of $X\smallsetminus X_0$ is a basis of $X\smallsetminus X_0$ iff 
	it is a minimal generating subset of $X\smallsetminus X_0$.
	\emph{\big[Here minimal generating subset $B$ of $X\smallsetminus X_0$ means there does
		not exist any proper subset of $B$ which can generate $X\smallsetminus X_0$.\big]}
\end{Th}

\begin{proof} Let us suppose $B$ be a basis of $X\smallsetminus X_0$. Then 
	$B$ generates $X\smallsetminus X_0$.
	Now $B$ being an orderly independent subset of $X\smallsetminus X_0$, if we 
	take an element $x\in B$ then $\forall\,y\in B\smallsetminus \{x\}$,
	$x$ and $y$ are orderly independent. Therefore $x\notin L(y)$, $\forall\,y\in B\smallsetminus \{x\}$. This shows that
	$B\smallsetminus \{x\}$ cannot generate $X\smallsetminus X_0$ and this holds
	for any $x\in B$. Therefore $B$ is a minimal generator of $X\smallsetminus X_0$.
	
	Conversely, suppose $B$ be a minimal generator of $X\smallsetminus X_0$. For 
	any two members $x,y\in B$ if $x\in L(y)$ 
	then by proposition \ref{ch5:p:L(x)}, $L(x)\subseteq L(y)$ $\Longrightarrow$ 
	$B\smallsetminus \{x\}$ also generates $X\smallsetminus X_0$, 
	which contradicts that $B$ is a minimal generator of $X\smallsetminus X_0$. 
	Again, if $y\in L(x)$ we get similar contradiction. So neither $x\in L(y)$ nor
	$y\in L(x)$ $\Longrightarrow$ $x,y$ are orderly independent.
	Arbitrariness  of $x,y$ shows that $B$ is an orderly independent subset of 
	$X\smallsetminus X_0$. Consequently, $B$ is a basis of $X\smallsetminus X_0$.
\end{proof}

\begin{Res} Every basis of $X\smallsetminus X_0$ is a maximal orderly independent
	subset of $X\smallsetminus X_0$. \emph{\big[Here maximal orderly independent 
		subset $B$ of $X\smallsetminus X_0$ means there does
		not exist any orderly independent subset of $X\smallsetminus X_0$ containing 
		$B$.\big]}\label{ch5:r:max}\end{Res}

\begin{proof}Let $B$ be a basis of $X\smallsetminus X_0$. Then for any 
	$x\in X\smallsetminus (B\cup X_0)$, $\exists$ $b\in B$ such that 
	$x\in L(b)$. This shows that $B\cup \{x\}$ is not orderly independent. Thus
	$B$ is maximal orderly independent in $X\smallsetminus X_0$.\end{proof}

Converse of above result is not true in general i.e maximal orderly independent
subset of $X\smallsetminus X_0$ may not be a basis of $X\smallsetminus X_0$.
For example, in the evs $\com{\X}{\theta}$ [discussed in \ref{e:topcthetax}]let us consider
the collection\\ \centerline{$\mathscr G:=\big\{A\in\com{\X}{\theta}:A\text{ consists of three distinct 
		elements of }\X\big\}$} Then $\mathscr G\subset\com{\X}{\theta}\smallsetminus\big\{\{\theta\}\big\}$.
Now we define a relation `$\thicksim$' on $\mathscr G$ by ``$A\thicksim B$ iff 
$A=\alpha B$ for some $\alpha\in\K^*$''. Then this relation becomes an equivalence relation on 
$\mathscr G$. Let us consider the subcollection $\mathscr H$ of $\mathscr G$
taking exactly one member from each equivalence class produced by the equivalence
relation `$\thicksim$'. Then $\mathscr H$ becomes an orderly
independent subset of $\com{\X}{\theta}\smallsetminus\big\{\{\theta\}\big\}$,
because any two elements $A,B\in\mathscr G$ are orderly dependent iff $A=\alpha B$
for some $\alpha\in\K^*$ and hence belong to the same 
equivalence class. For any member $C\in \com{\X}{\theta}\setminus (\mathscr H\cup [\com{\X}{\theta}]_0 )$
if $\cd C\geq 3$ then there exists a member $A_C\in \mathscr H$ and
$\alpha \in \K^*$ such that $\alpha A_C\subseteq C$.  If $\cd C=2$ then also 
there exists $\beta \in \K^*$ and $A_C\in \mathscr H$
such that $C\subseteq\beta A_C$. This shows that $\mathscr H\cup\{C\}$ is orderly
dependent. So we can say that $\mathscr H$ forms a maximal orderly independent set
in $\com{\X}{\theta}\smallsetminus\big\{\{\theta\}\big\}$.
But it does not generate $\com{\X}{\theta}\smallsetminus\big\{\{\theta\}\big\}$.
In fact, for any $D\in \com{\X}{\theta}$ with $\cd D=2$ there does not exist any
$A\in \mathscr H$ such that $D\in L(A)$, since each member of $L(A)$ contains 
three or more elements of $\X$. Hence $\mathscr H$ cannot be a basis of
$\com{\X}{\theta}\smallsetminus[\com{\X}{\theta}]_0$ although it is maximal orderly
independent \big[here note that $[\com{\X}{\theta}]_0=\big\{\{\theta\}\big\}$\big].

\begin{Rem} If $A$ is an orderly independent set in $X\smallsetminus X_0$ then for any 
	$a_1,a_2\in A$ with $a_1\neq a_2$ neither $a_1\in L(a_2)$ nor $a_2\in L(a_1)$.
	In other words, if $a_1\in L(a_2)$ for some $a_1,a_2\in A$ then $a_1=a_2$. Moreover
	any two elements of an orderly independent set $A$ must be incomparable with
	respect to the partial order `$\leq$' of the evs $X$; in fact, $x\leq y$ 
	$\Rightarrow$ $y\in L(x)$.\label{ch5:rm:oi}\end{Rem}

\begin{Lem} Let $A$ and $B$ be two bases of $X\smallsetminus X_0$. Then for
	any $a\in A$, there exists one and only one $b_a\in B$
	such that $L(a)= L(b_a)$.\label{ch5:l:card}\end{Lem}

\begin{proof} As $B$ is a basis of $X\smallsetminus X_0$, so for the member 
	$a\in A$, there must exist some $b\in B$
	such that $a\in L(b)$. Let us suppose, $\exists$ $b_1,b_2\in B$ such that 
	$a\in L(b_1)\cap L(b_2)$ $\Rightarrow$ $L(a)\subseteq L(b_1)\cap L(b_2)$ 
	[by proposition \ref{ch5:p:L(x)}] ------ $ (\ast) $. Again $a\in L(b_1)$ $\Rightarrow$
	$\exists$ $\alpha\in K^*$ and $p\in X_0$ such that $\alpha b_1+p\leq a$ ------ $ (\ast\ast) $. Now since $A$ is a basis, so for $b_1,b_2\in B$,
	$\exists$ $a_1,a_2\in A$ such that $b_1\in L(a_1)$ and $b_2\in L(a_2)$ 
	$\Rightarrow$ $L(b_1)\subseteq L(a_1)$
	and $L(b_2)\subseteq L(a_2)$ [by proposition \ref{ch5:p:L(x)}]. By $ (\ast) $, 
	$L(a)\subseteq L(a_1)$
	and $L(a)\subseteq L(a_2)$ $\Rightarrow$ $a\in L(a_1)$ and $a\in L(a_2)$. As
	$a,a_1,a_2$ are members of the basis $A$, so we can say that $a_2=a=a_1$
	[by above remark \ref{ch5:rm:oi}]. Therefore $b_1,b_2\in L(a)$ $\Rightarrow$ 
	$\exists$ $\alpha_1, \alpha_2 \in K^*$
	and $p_1,p_2\in X_0$ such that $\alpha_1 a+p_1\leq b_1$ and 
	$\alpha_2 a+p_2\leq b_2$ ------ $ (\ast\ast\ast) $. From $ (\ast\ast) $ and $ (\ast\ast\ast) $, we get 
	$b_2\geq \alpha_2 a+p_2\geq\alpha_2(\alpha b_1+p)+p_2=\alpha_2\alpha b_1+
	(\alpha_2 p+p_2)$. $\therefore$ $b_2\in L(b_1)$ [As $\alpha_2\alpha\in K^*$
	and $\alpha_2 p+p_2 \in X_0$]. Since $b_1, b_2$ are members of the basis $B$
	so $b_2=b_1$ [by above remark \ref{ch5:rm:oi}]. Thus there exists one and only one member (say) $b_a\in B$ 
	such that $a\in L(b_a)$ and also $b_a\in L(a)$ $\Rightarrow$ $L(a)\subseteq L(b_a)$
	and $L(b_a)\subseteq L(a)$ $\Rightarrow$ $L(a)=L(b_a)$.\end{proof}

\begin{Th}If $A$ and $B$ are two bases of $X\smallsetminus X_0$ then 
	$\cd A=\cd B$.\label{ch5:t:card}\end{Th}

\begin{proof} From the proof of the above lemma \ref{ch5:l:card} we can say 
	that for each $a\in A$, $\exists$ unique $b\in B$ such that
	$L(b)=L(a)$. This property creates a one to one correspondence between $A$ 
	and $B$. Hence the theorem.\end{proof}

This theorem motivates us to introduce the concept of dimension of an evs.

\begin{Def} For an evs $X$ we define \emph{dimension} of $X\smallsetminus X_0$
	as\\ \centerline{$\dim(X\smallsetminus X_0):=\cd B$, where $B$ is a basis of $X\smallsetminus X_0$.}
	Then we shall represent  \emph{dimension} of the evs $X$ as 
	$\dim X:=[\dim(X\smallsetminus X_0):\dim X_0]$. If $X_0=\{\theta\}$, dimension
	of $X_0$ will be taken as 0, since then $X_0$ has no basis [as vector space].
\end{Def}

\begin{Not}Theorem \ref{ch5:t:card} makes the above definition well-defined.
	From result \ref{ch5:r:basisiso} we can say that if $X$ and $Y$ are order-isomorphic
	evs then $\dim X=\dim Y$. Here by ``$\dim X=\dim Y$'' we mean $\dim(X\smallsetminus X_0)
	=\dim(Y\smallsetminus Y_0)$ as well as $\dim X_0=\dim Y_0$. However, converse
	of this is not true in general i.e there are evs $X,Y$ such that $\dim X=\dim Y$
	but $X,Y$ are not order-isomorphic. This will be clear in the next section, 
	when we shall compute the dimension of some particular evs.
\label{equality dimension}\end{Not}

\begin{Res} Let $X$ be an evs and $B$ be a basis of $X\smallsetminus X_0$.
	Then $\downarrow x\smallsetminus X_0\subseteq L(x)$, for each $x\in B$.
	\label{ch5:r:xinbasis}\end{Res}

\begin{proof} Let $x\in B$ and $y\in \downarrow x\smallsetminus X_0$. Since $B$ is a basis
	of $X\smallsetminus X_0$, $\exists$ $x_1\in B$ such that $y\in L(x_1)$ $\Rightarrow$
	$\exists$ $\alpha_1 \in  K^*$ and $p_1\in X_0$ such that 
	$\alpha_1x_1+p_1\leq y$ $\Rightarrow$ $\alpha_1x_1+p_1\leq x$ [$\because y\leq x$]
	$\Rightarrow$ $x\in L(x_1)$. Since $B$ is orderly independent and both $x,x_1\in B$,
	we can say by remark \ref{ch5:rm:oi} that $x=x_1$. Therefore $y\in L(x)$. Thus
	$\downarrow x\smallsetminus X_0\subseteq L(x)$, for each $x\in B$.
\end{proof}

This result reveals an important property of each member of a basis of 
$X\smallsetminus X_0$ which helps us to set up a precise domain of basic elements
of $X\smallsetminus X_0$. The collection of all $x$ in $X\smallsetminus X_0$ 
satisfying the property stated in the above result \ref{ch5:r:xinbasis} makes
our task of finding a basis of $X\smallsetminus X_0$
easier. To make this assertion precise let us consider the following :

For an evs $X$, let\\ 
\centerline{$Q(X):=\big\{x\in X\smallsetminus X_0: (\downarrow x\smallsetminus X_0)\subseteq L(x)\big\}$}
From result \ref{ch5:r:xinbasis} we can say that $B\subseteq Q(X)$, for any 
basis $B$ of $X\smallsetminus X_0$. It is thus enough to find any basis of 
$X\smallsetminus X_0$ within $Q(X)$. We call this set $Q(X)$ as \emph{feasible
	set} of $X$. At this point it is important to note that
$Q(X)$ may be empty; in fact, if for an evs $X$, $Q(X)=\emptyset$ then
such evs $X$ cannot have any basis (as we have claimed earlier).
We shall encounter such evs later. If for an evs $X$, $Q(X)\neq\emptyset$
then also $X$ may not have a basis. In fact, we shall prove shortly a theorem which will
characterise, in terms of $Q(X)$, when $X$ will have a basis. We now prove a 
lemma which will be useful in the sequel.

\begin{Lem}For an evs $X$, if $x\in Q(X)$ then for each $y\in\downarrow x\smallsetminus X_0$,
	$L(x)=L(y)$.\label{ch5:l:L(x)=L(y)}\end{Lem}

\begin{proof}$y\in\downarrow x\smallsetminus X_0$ $\Rightarrow$ $y\leq x$. So
	by proposition \ref{ch5:p:L(x)} we have $L(x)\subseteq L(y)$.
	Again by construction of $Q(X)$, $x\in Q(X)$ $\Rightarrow$ 
	$\downarrow x\smallsetminus X_0\subseteq L(x)$ $\Rightarrow$ $y\in L(x)$ 
	$\Rightarrow$ $L(y)\subseteq L(x)$ [by proposition \ref{ch5:p:L(x)}]. Thus 
	$L(x)=L(y)$.\end{proof}

The following theorem may be compared with the so-called `\emph{Replacement 
	theorem}' in the context of basis of a vector space.

\begin{Th} For an evs $X$, let $B$ be a basis of $X\smallsetminus X_0$ and 
	$x\in B$. Then for any $y\in\downarrow x\smallsetminus X_0$, 
	$\big(B\smallsetminus\{x\}\big)\cup \{y\}$ is also a basis of $X\smallsetminus X_0$.
\end{Th}

\begin{proof} Let $A=\big(B\smallsetminus\{x\}\big)\cup \{y\}$. As 
	$y\in\downarrow x\smallsetminus X_0$ and 
	$x\in B\subseteq Q(X)$ so by lemma \ref{ch5:l:L(x)=L(y)}, $L(x)=L(y)$. Therefore 
	$X\smallsetminus X_0=\displaystyle \bigcup_{z\in B}L(z)=\displaystyle 
	\bigcup_{z\in A}L(z)$ $\Rightarrow$ $A$ generates $X\smallsetminus X_0$. To
	show that $A$ is orderly independent it is sufficient to show that for any 
	$z\in B\smallsetminus\{x\}$, $z, y$ are orderly independent.
	If not, then for some $z_1\in B\smallsetminus\{x\}$ either $y\in L(z_1)$ or 
	$z_1\in L(y)$. Now if $y\in L(z_1)$ then by proposition \ref{ch5:p:L(x)},
	$L(y)\subseteq L(z_1)$ $\Rightarrow$ $x\in L(x)=L(y)\subseteq L(z_1)$ which 
	contradicts that $x,z_1$ are two members of the basis $B$.
	Again if $z_1\in L(y)$ then $z_1\in L(y)=L(x)$ which again contradicts that 
	$x,z_1$ are orderly independent. Thus it follows that $A$ is orderly independent.
\end{proof}

The above theorem makes it convenient to construct new basis from old one. The 
following theorem is the key to ensure the existence of a basis of an evs.

\begin{Th} An evs $X$ has a basis iff $Q(X)$ is a generator of $X\smallsetminus X_0$.
	\label{ch5:t:basis}\end{Th}

\begin{proof} Let us suppose $X$ has a basis $[B:B_0]$, where $B$ and $B_0$ 
	are bases of $X\smallsetminus X_0$ and $X_0$ respectively. Then by the result 
	\ref{ch5:r:xinbasis}, $B\subseteq Q(X)$. As $B$ is a generator of 
	$X\smallsetminus X_0$ so $Q(X)$ is also a generator of $X\smallsetminus X_0$.
	
	Conversely, suppose $Q(X)$ is a generator of $X\smallsetminus X_0$ 
	$\Rightarrow$ $Q(X)\neq \emptyset$. We now give a relation $\sim$ in $Q(X)$ as 
	follows : For $x,y \in Q(X)$, we say $x\sim y$ $\Leftrightarrow$ $L(x)=L(y)$.
	Then obviously this becomes an equivalence relation on $Q(X)$. Let us 
	consider a collection taking exactly one representative from each equivalence
	class and denote this collection as $B$. Then $B\subseteq Q(X)\subseteq X\smallsetminus X_0$.
	Also $x,y\in B$ with $x\neq y$ $\Leftrightarrow$ $x,y\in Q(X)$ and $L(x)\neq L(y)$. We claim
	that $B$ is a basis of $X\smallsetminus X_0$. Let $z\in X\smallsetminus X_0$
	$\Rightarrow$ $\exists$ $x_z\in Q(X)$ [$\because Q(X)$ is a generator] such that 
	$z\in L(x_z)$ $\Rightarrow$ $\exists$ an element $x'_z\in B$ such that 
	$L(x_z)=L(x'_z)$ and hence $z\in L(x'_z)$ $\Rightarrow$ $B$ generates $X\smallsetminus X_0$. 
	Now we have to show that $B$ is orderly independent. Suppose not, then $\exists$ two distinct elements 
	$x_1,x_2\in B$ such that they are orderly dependent. So without loss 
	of generality we can think that $x_1\in L(x_2)$. Now $x_1\in L(x_2)$ 
	$\Rightarrow$ $\exists$ $\alpha \in \K^*$ and $p\in X_0$ such that 
	$\alpha x_2+p\leq x_1$. Since $x_1\in Q(X)$ [as $B\subseteq Q(X)$]  
	and $\alpha x_2+p\in\downarrow x_1\smallsetminus X_0$, using lemma \ref{ch5:l:L(x)=L(y)}
	we can say that $L(\alpha x_2+p)=L(x_1)$ and then by proposition \ref{ch5:p:L(x)}
	we have $L(x_2)=L(\alpha x_2+p)=L(x_1)$ ------ which contradicts that 
	$x_1,x_2$ are two distinct elements of $B$. Thus $B$ becomes a basis of $X\smallsetminus X_0$.
	Let us take any basis $B_0$ of $X_0$. Then $[B:B_0]$ becomes a basis of $X$.
\end{proof}

From the proof of the above theorem it is clear that the hypothesis of $Q(X)$
being a generator of $X\smallsetminus X_0$ is only used to justify that $B$,
constructed using the equivalence relation $\sim$ within $Q(X)$, is a generator
of $X\smallsetminus X_0$; to ensure the orderly independence of $B$, the structure
of $Q(X)$ is enough. Thus we can conclude that for any evs $X$, $Q(X)$ (if nonempty)
always contains an orderly independent set like $B$ (as constructed in the proof
of the above theorem \ref{ch5:t:basis}). This orderly independent set is also
a maximal orderly independent set in $Q(X)$. In fact, if $D$ be another orderly 
independent set in $Q(X)$ such that $B\subset D$ then for any $x\in D$, $\exists\,z\in B$
such that $L(x)=L(z)$ $\Rightarrow$ $x=z$ [$\because x,z\in D$ and $D$ is orderly
independent] and hence $x\in B$. Thus $B=D$. Summarising all these facts we get
the following theorem.

\begin{Th} For an evs $X$, if $Q(X)\neq\emptyset$ then it contains a maximal 
	orderly independent set.\label{ch5:t:max}\end{Th}

The next theorem is useful in finding a basis of $X\smallsetminus X_0$, for any
evs $X$.

\begin{Th}For an evs $X$, every maximal orderly independent set of $Q(X)$ is a
	basis of $X\smallsetminus X_0$, provided $Q(X)$ generates $X\smallsetminus X_0$. 
	\label{ch5:t:conmax}\end{Th}

\begin{proof} Let $B$ be a maximal orderly independent set in $Q(X)$. Since 
	$Q(X)$ generates $X\smallsetminus X_0$, for any $x\in X\smallsetminus X_0$, 
	$\exists\,z\in Q(X)$ such that $x\in L(z)$. If $z\in B$ we are done. If $z\notin B$
	then $B$ being a maximal orderly independent set in $Q(X)$, $B\cup\{z\}$ is 
	orderly dependent. So $\exists\,b\in B$ such that either $z\in L(b)$ or $b\in L(z)$.
	If $z\in L(b)$ then by proposition \ref{ch5:p:L(x)}, $x\in L(z)\subseteq L(b)$.
	If $b\in L(z)$, $\exists\,\alpha\in\K^*$ and $p\in X_0$ such that $b\geq\alpha z+p$.
	Then $b\in B\subseteq Q(X)$ $\Rightarrow$ $L(z)=L(b)$ [by lemma \ref{ch5:l:L(x)=L(y)}]
	$\Rightarrow$ $x\in L(b)$. Thus $B$ generates $X\smallsetminus X_0$. Consequently
	$B$ is a basis of $X\smallsetminus X_0$.
\end{proof}

The above theorem \ref{ch5:t:conmax} shows the converse of the result \ref{ch5:r:max} to some extent; as we 
have explained, just after the result \ref{ch5:r:max}, through the example of $\com{\X}{\theta}$
that every maximal orderly independent subset of $X\smallsetminus X_0$ need 
not be a basis of $X\smallsetminus X_0$, the above theorem \ref{ch5:t:conmax} shows that every 
maximal orderly independent subset of $Q(X)$ [but \emph{not only} of $X\smallsetminus X_0$]
becomes a basis of $X\smallsetminus X_0$, provided of course $Q(X)$ generates
$X\smallsetminus X_0$ [note that the necessity of $Q(X)$ being a generator of
$X\smallsetminus X_0$ is the principal key for an evs $X$ to have a basis]. From
remark \ref{ch5:rm:oi} we may recall one more point that while finding a basis
of $X\smallsetminus X_0$, we have to gather only suitable incomparable 
elements from $Q(X)$. In this context it should also be noted that any two 
elements of $Q(X)$ need not be orderly independent. In fact, for any $x\in Q(X)$
if $y\in\downarrow x\smallsetminus X_0$ then also $y\in Q(X)$ [see result 
\ref{ch5:r:QX}(ii)]. Clearly this $x,y$ are orderly dependent, since $L(x)=L(y)$.

\begin{Res}If $X$ and $Y$ are order-isomorphic then $Q(X)$ and $Q(Y)$ are in 
	a one-to-one correspondence.\label{ch5:r:cardQX}\end{Res}

\begin{proof} Let $\phi:X\longrightarrow Y$ be an order-isomorphism. We now 
	show that $\phi(Q(X))=Q(Y)$.
	Let $x_0\in Q(X)$ $\Rightarrow$ $\downarrow x_0\smallsetminus X_0\subseteq 
	L(x_0)$. Also let $y\in \downarrow \phi(x_0)\smallsetminus Y_0$
	$\Rightarrow$ $y\leq \phi(x_0)$ and $y\notin Y_0$ $\Rightarrow$ 
	$\phi^{-1}(y)\leq x_0$ and $\phi^{-1}(y)\notin X_0$
	$\Rightarrow$ $\phi^{-1}(y)\in \downarrow x_0\smallsetminus X_0\subseteq 
	L(x_0)$ $\Rightarrow$ $\exists$ $\alpha \in \K^*$ and
	$p\in X_0$ such that $\alpha x_0+p\leq \phi^{-1}(y)$ $\Rightarrow$ 
	$\alpha\phi(x_0)+\phi(p)\leq y$ $\Rightarrow y\in L(\phi(x_0))$ [$\because$ 
	$\phi(p)\in Y_0$]. Therefore $\downarrow\phi(x_0)\smallsetminus Y_0 
	\subseteq L(\phi(x_0))$ $\Rightarrow$ $\phi(Q(X))\subseteq Q(Y)$. Similarly 
	we can say that $\phi^{-1} (Q(Y))\subseteq Q(X)$ [$\because$ $\phi^{-1}$ is 
	an order-isomorphism from $Y$ onto $X$] $\Rightarrow$ $Q(Y)\subseteq \phi(Q(X))$.\\
	$\therefore$ $\phi(Q(X))=Q(Y)$. Thus $Q(X)$ and $Q(Y)$ are in a one-to-one 
	correspondence.\end{proof}

\begin{Res}{\em(i)} If $x\in Q(X)$ then for any $\alpha\in \K^*$ and $p\in X_0$, 
	$\alpha x+p\in Q(X)$ i.e $Q(X)$ is closed under dilation and translation by
	primitive elements.\\
	{\em(ii)} If $x\in Q(X)$ then $\downarrow x\smallsetminus X_0\subseteq Q(X)$
	i.e $\downarrow Q(X)\smallsetminus X_0\subseteq Q(X)$.
	\label{ch5:r:QX}\end{Res}

\begin{proof}(i) $x\in Q(X)$ $\Rightarrow $ $\downarrow x\smallsetminus X_0\subseteq L(x)$.
	We now show that $\downarrow(\alpha x+p)\smallsetminus X_0
	\subseteq L(\alpha x+p)$. Let $y\in \downarrow (\alpha x+p)\smallsetminus X_0$
	$\Rightarrow$ $y\leq \alpha x+p$ and $y\notin X_0$ $\Rightarrow$ 
	$\alpha^{-1}(y-p)\leq x$ and $y\notin X_0$ $\Rightarrow$ 
	$\alpha ^{-1}(y-p)\in \downarrow x\smallsetminus X_0$ $\Rightarrow$ 
	$\alpha^{-1}(y-p)\in L(x)$ $\Rightarrow$ $\exists$ $\beta\in \K^*$ and 
	$q\in X_0$ such that $\beta x+q\leq \alpha^{-1}(y-p)$ $\Rightarrow$ 
	$\alpha(\beta x+q)+p\leq y$ $\Rightarrow$ $\alpha\beta x+\alpha q+p\leq y$ 
	$\Rightarrow$ $y\in L(x)$ [$\because\alpha q+p\in X_0$] $\Rightarrow$ 
	$y\in L(\alpha x+p)$ [$\because L(\alpha x+p)=L(x)$, by proposition \ref{ch5:p:L(x)}].
	Therefore $\alpha x+p\in Q(X)$.
	
	(ii) Let $y\in\downarrow x\smallsetminus X_0$. Then by lemma \ref{ch5:l:L(x)=L(y)},
	$L(x)=L(y)$. Now for each $z\in\downarrow y\smallsetminus X_0$ we have $z\leq y\leq x$
	with $z\notin X_0$ $\Rightarrow$ $L(z)=L(x)$ [by lemma \ref{ch5:l:L(x)=L(y)}]
	$\Rightarrow$ $z\in L(x)=L(y)$. Thus $\downarrow y\smallsetminus X_0\subseteq L(y)$.
	Consequently, $y\in Q(X)$ and hence $\downarrow x\smallsetminus X_0\subseteq Q(X)$.
\end{proof}

As we have explained in remark \ref{ch5:rm:suboi} regarding orderly independence
in a subevs of an evs, the theory of basis of a subevs does not behave nicely 
like the theory of basis of a subspace of a vector space. However we have the 
following theorems and examples which reveal some technical aspects of
dimension theory of evs.

\begin{Th}Every evs contains a subevs of dimension $[1:0]$.\end{Th}

\begin{proof} Let $X$ be an evs over $\K$ and 
	$B(x):=\left\{\displaystyle\sum_{i=1}^n\alpha_i x:\alpha_i\in\K,n\in\N\right\}$,
	where $x\in\uparrow\theta\smallsetminus\{\theta\}$. Then for any $\alpha ,\beta \in \K$ and any
	$\displaystyle \sum_{i=1}^n \alpha_i x$, $\displaystyle \sum_{j=1}^m \beta_j x \in B(x)$
	we have $\alpha \displaystyle \sum_{i=1}^n \alpha_i x+\beta \displaystyle \sum_{j=1}^m \beta_j x$
	$=\displaystyle \sum_{i=1}^n\alpha \alpha_i x+\displaystyle \sum_{j=1}^m \beta\beta_j x\in B(x)$.
	Also $[B(x)]_0=\{\theta\}=B(x)\cap X_0$ and for any $y\in B(x)$, $\theta\leq y$
	[$\because \theta \leq x$]. So $B(x)$ forms a subevs of $X$ for any 
	$x\in\uparrow\theta\smallsetminus\{\theta\}$.
	In this case $\{x\}$ forms a basis of $B(x)\smallsetminus[B(x)]_0$. In 
	fact, for any $\displaystyle \sum_{i=1}^n \alpha_i x\in B(x)$,
	$\alpha_jx+\theta\leq \displaystyle \sum_{i=1}^n \alpha_i x$ for some 
	$j\in \{1,2,\dots,n\}$ for which $\alpha_j\neq 0$.
	Again any singleton set consisting of a non-zero element is always orderly independent.
	So we can say that $\dim B(x)=[1:0]$.\end{proof}

The following example shows that corresponding to any cardinal $\alpha$, there
exists an evs of dimension $[\alpha:0]$.

\begin{Ex} For any cardinal number $\alpha$, let us consider the evs 
	$[0,\infty)^\alpha$, discussed in Example \ref{e:prod}. Let us take a set $I$ such that $\cd I=\alpha$. We now
	show that $B:=\{e_i:i\in I\}$ is a basis of $[0,\infty)^\alpha$, where 
	$e_i=(\delta^i_j)_{j\in I}$ and 
	$\delta^i_j=\begin{cases}1,\text{ when }i=j\\ 0,\text{ when }i\neq j\end{cases}$\\
	For any $x\in [0,\infty)^\alpha\smallsetminus\big[[0,\infty)^\alpha\big]_0$ 
	\Big[here $\big[[0,\infty)^\alpha\big]_0=\{\theta\}$ and $\theta=(z_j)_{j\in I}$,
	where $z_j=0$, $\forall\,j\in I$\Big] with representation $x=(x_j)_{j\in I}$, 
	$\exists$ $p\in I$ such that $x_p\neq0$ $\Rightarrow$ $x_pe_p\leq x$ 
	[$\because x_j\geq0$, $\forall\,j\in I$]
	$\Rightarrow$ $x_pe_p+\theta\leq x$ $\Rightarrow$ $x\in L(e_p)$ $\Rightarrow$
	$B$ generates $[0,\infty)^\alpha\smallsetminus\big[[0,\infty)^\alpha\big]_0$.
	Now clearly any two members of $B$ are orderly independent in 
	$[0,\infty)^\alpha\smallsetminus\big[[0,\infty)^\alpha\big]_0$.
	This shows that $B$ is a basis of $[0,\infty)^\alpha\smallsetminus\big[[0,\infty)^\alpha\big]_0$.
	Therefore $\dim [0,\infty)^\alpha=[\alpha:0]$, since $\cd B=\cd I=\alpha$ and 
	$\dim\big[[0,\infty)^\alpha\big]_0=\dim\{\theta\}=0$.
	\label{ch5:e:alpha}\end{Ex}

Thus for any two cardinal numbers $\alpha,\beta$ with $\alpha\neq\beta$,
$[0,\infty)^\alpha$ and $[0,\infty)^\beta$ cannot be order-isomorphic, since they are of different dimension.
We now show that for any two cardinal numbers $\alpha,\beta$, there exists an evs
of dimension $[\alpha:\beta]$. For this we need the following theorem first.

\begin{Th}For an evs $X$ and a vector space $V$, both being over the common field $\K$,
	the evs $Y:=X\times V$ has a basis iff the evs $X$ has a basis \emph{[The evs 
		$X\times V$ is discussed in example \ref{ch3:e:xxe}]}. 
	Also $\dim(X\times V)=[\dim (X\smallsetminus X_0):\dim X_0+\dim V]$.
	\label{ch5:t:xe}\end{Th}

\begin{proof}Let $X$ has a basis. We first show that $A:=\{(b,\theta_V): b\in B\}$
	is a basis of $Y\smallsetminus Y_0$,
	where $B$ is a basis of $X\smallsetminus X_0$ and $\theta_V$ is the identity 
	of $V$. As $B$ is orderly independent in $X\smallsetminus X_0$ we can say that any two members of $A$ are orderly
	independent $\Rightarrow$ $A$ is an orderly independent set in $Y\smallsetminus Y_0$.
	Let $(x,v)\in Y\smallsetminus Y_0$ $\Rightarrow$ $x\in X\smallsetminus X_0$
	[$\because Y_0=X_0\times V$]. Since $B$ generates $X\smallsetminus X_0$, for
	this $x$, $\exists$ $b\in B$ such that $x\in L(b)$
	$\Rightarrow$ $\alpha b+p\leq x$ for some $\alpha \in \K^*$ and $p\in X_0$ $\Rightarrow$
	$\alpha (b,\theta_V)+(p,v)=(\alpha b+p,v)\leq(x,v)$ and $(p,v)\in[X\times V]_0$
	$\Rightarrow$ $(x,v)\in L\big((b,\theta_V)\big)$ $\Rightarrow$
	$A$ is a generator of $Y\smallsetminus Y_0$. So $A$ becomes a basis of 
	$Y\smallsetminus Y_0$. Consequently $Y$ has a basis. Now 
	$\dim(Y\smallsetminus Y_0)=\cd A=\cd B=\dim (X\smallsetminus X_0)$ and
	$\dim[X\times V]_0=\dim (X_0\times V)=\dim X_0+\dim V$. Therefore 
	$\dim(X\times V)=[\dim (X\smallsetminus X_0):\dim X_0+\dim V]$.
	
	Conversely, suppose $Y:=X\times V$ has a basis. Let $B$ be a basis of 
	$Y\smallsetminus Y_0$. Now consider $B':=\{x: (x,v_x)\in B\text{ for some }v_x\in V\}$.
	Then $x\in B'$ $\Rightarrow$ $x\notin X_0$. Therefore $B'\subseteq X\smallsetminus X_0$.
	We now show that $B'$ forms a basis of $X\smallsetminus X_0$.
	For any $z\in X\smallsetminus X_0$, $(z,\theta_V)\in Y\smallsetminus Y_0$. As
	$B$ is a basis of $Y\smallsetminus Y_0$, $\exists$ $(x,v_x)\in B$ such that
	$\alpha (x,v_x)+(p,v)\leq (z,\theta_V)$ for some $\alpha \in \K^*$ and 
	$(p,v)\in[X\times V]_0=X_0\times V$ $\Rightarrow$ $(\alpha x+p,\alpha v_x+v)\leq (z,\theta_V)$
	$\Rightarrow$ $\alpha x+p\leq z$ $\Rightarrow$ $z\in L(x)$. So $B'$ generates
	$X\smallsetminus X_0$. If two members of $B'$ say $x',z'$ are
	orderly dependent then without loss of generality we can take $x'\in L(z')$ 
	$\Rightarrow$ $\exists$ $\alpha\in \K^*$ and $p\in X_0$ such that
	$\alpha z'+p\leq x'$ $\Rightarrow$ $\alpha (z',v_{z'})+(p,v_{x'}-\alpha v_{z'})\leq(x',v_{x'})$
	$\Rightarrow$ $(x',v_{x'})$ and $(z',v_{z'})$ are orderly dependent in 
	$Y\smallsetminus Y_0$. Therefore we can say that $B'$ is orderly independent 
	in $X\smallsetminus X_0$ as $B$ is orderly independent in $Y\smallsetminus Y_0$.
	So $B'$ becomes a basis of $X\smallsetminus X_0$. Consequently, $X$ has a basis.
\end{proof}

\begin{Ex} For any two cardinal numbers $\alpha, \beta$ there exists an evs 
	$X$ such that $\dim X=[\alpha:\beta]$. For example, if we consider the evs 
	$X:=Y\times E$, where $Y$ is an evs whose dimension is $[\alpha:0]$ (existence 
	of such evs has been established in example \ref{ch5:e:alpha}) and
	$E$ is a vector space with dimension $\beta$, then by above theorem 
	\ref{ch5:t:xe} $\dim X=[\alpha:\beta]$.
\end{Ex}

\begin{Th} Let $X$ be an evs whose dimension is $[\alpha:\beta]$. Also let $\gamma$ and $\delta$ be two cardinal
	numbers such that $\gamma \leq \alpha$ and $\delta\leq \beta$. Then $\exists$ a subevs $Y$ of $X$ such that
	$\dim Y=[\gamma:\delta]$.\end{Th}

\begin{proof} Let $B$ be a basis of $X\smallsetminus X_0$. Then $\cd B=\alpha$.
	Since $\gamma\leq\alpha$, there exists $C\subseteq B$ such that $\cd C=\gamma$.
	For each $c\in C$ we choose one element $p_c\in P_c$ and fix it. \\
	\textbf{\underline{Case 1} :} If $\delta<\gamma$ then $\exists$ 
	$E\subsetneqq C$ such that $\cd E=\delta$. Consider the set\\ 
	\centerline{$D:=E\cup\{c-p_c:c\in C\smallsetminus E\}$}
	Since $C$ is orderly independent it follows that $\cd D=\cd C=\gamma$. As  
	$L(c-p_c)=L(c)$ it follows that $D$ is an orderly independent set in $X\smallsetminus X_0$.
	Also consider for any $d\in D$, $q_d=p_d$ if $d\in E$ otherwise $q_d=\theta$.
	Then there exists a subspace $W$ of the vector space $X_0$  such
	that $q_d\in W$, $\forall\,d\in D$ and $\dim W=\delta$.\\
	\textbf{\underline{Case 2} :} If $\gamma \leq \delta$ then consider $D=C$ 
	and $q_d=p_d$, $\forall\,d\in D$. Then also there exists a subspace $W$ of 
	$X_0$ such that $q_d\in W$, $\forall\,d\in D$ and $\dim W=\delta$. \\
	Thus for both cases we get\\ (i) an orderly independent set $D$ in 
	$X\smallsetminus X_0$ whose cardinality is $\gamma$.\\
	(ii) a subspace $W$ of $X_0$ such that $q_d\in W$ where $q_d<d$,\footnote{Here
		the notation `$q_d<d$' is used to mean that $q_d\leq d$ but $q_d\neq d$.} 
	$\forall\,d\in D$ and $\dim W=\delta$.\\
	Now we consider the set $$G(D):=\left\{\displaystyle \sum_{i=1}^n\alpha_i d_i+p: \alpha_i\in \K,
	d_i\in D,p\in W,n\in\N\right\}$$
	\textbf{\underline{Step 1} :} In this step we will prove that $G(D)$ becomes
	a subevs of $X$ with  $D\subseteq G(D)$ and $[G(D)]_0=W$. \\
	For any $d\in D$, $d=1.d+\theta\in G(D)$ $\Rightarrow$ $D\subseteq G(D)$. Also for any $p\in W$,
	$0.d+p\in G(D)$ $\Rightarrow$ $W\subseteq G(D)$. For any two elements
	$x=\displaystyle \sum_{i=1}^m\alpha_id_i+p$, $y=\displaystyle \sum_{j=1}^n\beta_jd_j+q$ in $G(D)$  and
	any two scalars $\alpha,\beta$,
	$\alpha x+\beta y=\displaystyle \sum_{i=1}^m \alpha \alpha_id_i+\displaystyle
	\sum_{j=1}^n\beta\beta_j d_j+(\alpha p+\beta q)\in G(D)$
	[as $W$ is a subspace]. Let $y\in [G(D)]_0$. Then $y$ is a minimal element of
	$G(D)$. As $y\in G(D)$, $y$ can be written as $y=\displaystyle \sum_{i=1}^n\alpha_i d_i+p$.
	Our claim is that all $\alpha_i=0$. If not, there exists $j\in \{1,2,\dots,n\}$
	such that $\alpha_j\neq 0$. Then there exists $q_{d_j}\in W$ such that
	$q_{d_j}< d_j$ $\Rightarrow$ $\displaystyle \sum_{i=1}^n \alpha_i q_{d_i}+p< y$
	which contradicts that $y\in [G(D)]_0$, as
	$\displaystyle \sum_{i=1}^n \alpha_i q_{d_i}+p\in W\subseteq G(D)$.
	So all $\alpha_i=0$. Therefore $y=p\in W$ $\Rightarrow$
	$[G(D)]_0\subseteq W\subseteq G(D)\cap X_0$. Therefore $[G(D)]_0=G(D)\cap X_0=W$
	[by Note \ref{n:subevs}]. Also for any $x=\displaystyle \sum_{i=1}^n\alpha_i d_i+p\in G(D)$,
	$\displaystyle \sum_{i=1}^n \alpha_i q_{d_i}+p\in W=[G(D)]_0$ such that 
	$x\geq\displaystyle \sum_{i=1}^n \alpha_i q_{d_i}+p$. Thus it follows that $G(D)$
	is a subevs of $X$.\\
	\textbf{\underline{Step 2} :} In this step we shall show that $D$ is a basis
	of $G(D)\smallsetminus [G(D)]_0$. Since $D$ is an orderly independent subset of 
	$X\smallsetminus X_0$ and $G(D)$ is a subevs of $X$ containing $D$, by remark
	\ref{ch5:rm:suboi} we can say that $D$ is orderly independent in $G(D)\smallsetminus [G(D)]_0$.
	Now let $y\in G(D)\smallsetminus [G(D)]_0$. Then $y$ can be written as 
	$y=\displaystyle \sum_{i=1}^n\alpha_id_i+p$, where not all $\alpha_i=0$. Let 
	$\alpha_j\neq 0$. Then $\alpha_j d_j+\left(\displaystyle\sum_{\substack{i=1\\i\neq j}}^n
	\alpha_iq_{d_i}+p\right)\leq y$. As $\left(\displaystyle\sum_{\substack{i=1\\
			i\neq j}}^n\alpha_iq_{d_i}+p\right)\in W=[G(D)]_0$,
	so $y\in L(d_j)$ in $G(D)\smallsetminus [G(D)]_0$. Thus $D$ becomes a basis of $G(D)\smallsetminus [G(D)]_0$. \\
	Therefore $\dim G(D)=[\cd D:\dim W]=[\gamma:\delta]$.\end{proof}

\section{Computation of basis and dimension of some evs}
\markright{Computation of basis and dimension of some evs}

In this section we shall discuss the existence of basis of some particular evs
and thereby compute their dimensions. We show that there are evs which do not
have basis.

\begin{Th} Let $X$ be a single-primitive comparable topological evs. Then $X$
	has a basis and $\dim X=[1:\dim X_0]$.\end{Th}

\begin{proof}Since $X$ is single-primitive, for each $z\in X$ let us write 
	$P_z=\{p_z\}$. Let $x\in \uparrow \theta$ with $x\neq \theta$. Then $P_x=\{p_x\}=\{\theta\}$.
	Now for $y\in X\smallsetminus X_0$, $y-p_y\in \uparrow \theta$. Then $X$ being
	comparable evs, $x$ and $y-p_y$ are comparable as $P_x=P_{y-p_y}=\{\theta\}$. 
	If $x\leq y-p_y$ then $x+p_y\leq y$ $\Rightarrow$ $y\in L(x)$.
	If $x>y-p_y$ our claim is that there exists $\alpha \in \K^*$ such that 
	$\alpha x\leq y-p_y$ with $|\alpha| <1$. For, otherwise we can choose a sequence
	$\{\alpha_n\}$ in $\K^*$ such that $y-p_y<\alpha_n x\ \forall\,n\in\N$ and
	$\alpha_n\rightarrow 0$ as $n\to\infty$. Since $X$ is a topological evs we then
	have $y-p_y\leq \theta$ [taking limit $n\to\infty$] ------ a contradiction. 
	So there must exist one $\alpha \in \K^*$ such that $\alpha x\leq y-p_y$ 
	$\Rightarrow$ $y\in L(x)$. Thus $L(x)=X\smallsetminus X_0$. Clearly $\{x\}$ is
	orderly independent. Therefore $\{x\}$ is a basis of $X\smallsetminus X_0$.
	Consequently $X$ has a basis and $\dim X=[1:\dim X_0]$.\end{proof}

As the evs $[0,\infty)\times V$ is a single primitive comparable evs by above theorem we can say that $\dim([0,\infty)\times V)=[1:\dim V]$, for any Hausd\"{o}rff
topological vector space $V$. So in particular, if $V=\{\theta\}$ then the resulting
evs is order-isomorphic to $[0,\infty)$ and hence $\dim[0,\infty)=[1:0]$.
We have shown in the previous section that
$\dim[0,\infty)^\alpha=[\alpha:0]$, for any cardinal $\alpha$. This can also
be justified from the following more general example.

\begin{Ex} Let $\{X_i:i\in I\}$ be an arbitrary collection of exponential  vector 
	spaces, over the common field $\K$, each having a basis. Let $B_i$
	be a basis of $X_i\smallsetminus [X_i]_0$, $\forall\,i\in I$. Consider the 
	product evs $X:=\displaystyle \prod_{i\in I}X_i$ [see Example \ref{e:prod}]. Then $X_0=\displaystyle \prod_{i\in I}[X_i]_0$.
	For any $j \in I$ consider the set $D_j:=\displaystyle \prod_{i \in I} C_i$, 
	where $C_i:=\begin{cases} \{\theta_{X_i}\},\text{ when }i\neq j\\
	B_j,\text{ when }i=j \end{cases}$. Here $\theta_{X_i}$ is the identity in $X_i$.
	Then $D_j\subseteq X\smallsetminus X_0$, $\forall\,j\in I$. Let 
	$D:=\displaystyle\bigcup_{j\in I}D_j$. Then $D\subseteq X\smallsetminus X_0$.
	Now two different members in different $D_i$ are orderly independent. As each
	$B_i$ is a basis of $X_i$, so two different members
	of one $D_i$ are orderly independent. Thus any two different members of $D$ are 
	orderly independent and hence $D$ is orderly independent in $X\smallsetminus X_0$.
	We now show that $D$ is a basis of $X\smallsetminus X_0$. For any 
	$x=(x_i)_{i \in I}\in X\smallsetminus X_0$, $\exists$ some $k \in I$ such that
	$x_k\in X_k \smallsetminus [X_k]_0$ $\Rightarrow$ $\exists$ $b_k\in B_k$,
	$\alpha_k \in \K^*$ and $p_k\in [X_k]_0$ such that $\alpha_k b_k+p_k\leq x_k$.
	Now for $i \neq k$, $\exists$ $p_i \in[X_i]_0$ such that $p_i\leq x_i$. Let
	$b=(b_i)_{i \in I}$, where $b_i=\theta_{X_i}$ for $i \neq k$ and $p=(p_i)_{i \in I}\in X_0$.
	Then $\alpha_k b+p=(\alpha_k b_i+p_i)_{i \in I}\leq (x_i)_{i \in I}=x$ and
	$b\in D_k\subset D$ $\Rightarrow$ $x\in L(b)$. This shows that $D$ generates 
	$X\smallsetminus X_0$ and hence is a basis of $X\smallsetminus X_0$. Consequently, 
	$X$ has a basis and $\dim X=[\cd D:\dim X_0]$.  
	
	If $I$ be finite then $\cd D=\displaystyle\sum_{i\in I}\cd{D_i}
	=\sum_{i\in I}\cd{B_i}=\sum_{i\in I}\dim\big(X_i\smallsetminus[X_i]_0\big)$
	and $\dim X_0=\displaystyle\sum_{i\in I}\dim[X_i]_0$. For any four cardinal 
	number $\alpha,\beta,\gamma,\delta$ if we use the notation 
	$[\alpha+\gamma:\beta+\delta]=[\alpha:\beta]+[\gamma:\delta]$ then we can write
	$$\displaystyle\dim\prod_{i\in I}X_i=\left[\sum_{i\in I}\dim\left(X_i\smallsetminus[X_i]_0\right)
	:\sum_{i\in I}\dim[X_i]_0\right]=\sum_{i\in I}\big[\dim\left(X_i\smallsetminus[X_i]_0\right)
	:\dim[X_i]_0\big]$$
	
	If $I$ be infinite then also we get the similar expression
	as above, provided the sums (over $I$) be properly defined.
	
	If all $X_i$'s are same, say $X_i=Y,\ \forall\,i\in I$ and $\cd I=\alpha$ then
	we have\\ \centerline{$\dim(Y^\alpha)=[\alpha\cdot\dim(Y\smallsetminus Y_0):\alpha\cdot\dim Y_0]$}
	Thus it follows that for any cardinal $\alpha$, $\dim[0,\infty)^\alpha=[\alpha:0]$,
	since $\dim[0,\infty)=[1:0]$.
\end{Ex}

\begin{Th} For every Hausd\"{o}rff topological vector space $\X$, 
	$\com{\X}{}$ [discussed in \ref{CX}] has a basis.\label{ch5:t:dimcx}\end{Th}

\begin{proof}Let us consider the relation `$\sim$' on $\X\smallsetminus \{\theta\}$, defined as\\
	\centerline{$x\sim y$ $\Leftrightarrow$ $\exists$ $\alpha\in \K^*$ such that $x=\alpha y$}
	Then `$\sim$' becomes an equivalence relation on $\X\smallsetminus\{\theta\}$.
	Let us construct a set $\X'$ taking exactly one representative
	from each equivalence class relative to `$\sim$' and consider the set\\
	\centerline{$\mathscr N:=\big\{\{\theta,x\}: x\in \X'\big\}$}
	We now show that $\mathscr N$ becomes a basis of $\com{\X}{}\smallsetminus [\com{\X}{}]_0$.
	If $A\in \com{\X}{}\smallsetminus [\com{\X}{}]_0$,
	then there must exist two elements $x,y$ of $\X$ with $\{x,y\}\subseteq A$ and $x\neq y$.
	Then $\{\theta,x-y\}+\{y\}=\{x,y\}\subseteq A$.
	Now $x-y\in\X\smallsetminus\{\theta\}$ $\Rightarrow$ $\exists$ $z\in\X'$ and 
	$\alpha \in \K^*$ such that $x-y=\alpha z$. So we can write
	$\alpha\{\theta,z\}+\{y\}\subseteq A$ $\Rightarrow$ $A\in L(\{\theta,z\})$. 
	Therefore $\mathscr N$ generates $\com{\X}{}\smallsetminus[\com{\X}{}]_0$.
	We now show that $\mathscr N$ is an orderly independent set in $\com{\X}{}\smallsetminus [\com{X}{}]_0$. For
	any two elements $\{\theta,x\}$ and $\{\theta,y\}$ in $\mathscr N$, if $\{\theta,x\}\in L(\{\theta,y\})$
	then $\exists$ $\alpha \in K^*$ such that $\alpha\{\theta,y\}+\{z\}\subseteq \{\theta,x\}$ for
	some $z\in \X$ $\Rightarrow$ $\{z,\alpha y+z\}=\{\theta,x\}$ [$\because z\neq\alpha y+z$]
	$\Rightarrow$ either $z=\theta$ or $z=x$. If $z=\theta$ then $\alpha y=x$ 
	which means that $x,y$ belong to the same equivalence class relative to 
	`$\sim$' and hence $\{\theta,x\},\{\theta,y\}$ cannot be two distinct elements
	of $\mathscr N$ ------ which is not the case. If $z=x$ then
	$\alpha y+x=\theta$ $\Rightarrow$ $x=-\alpha y$ which again leads to the same
	contradiction. This proves that any two elements of
	$\mathscr N$ are orderly independent. Therefore $\mathscr N$ is orderly 
	independent in $\com{\X}{}\smallsetminus [\com{X}{}]_0$ and
	hence becomes a basis of $\com{\X}{}\smallsetminus [\com{\X}{}]_0$. Consequently,
	$\com{\X}{}$ has a basis and $\dim\com{\X}{}=[\cd{\mathscr N}:\dim\X]$.
\end{proof}

\begin{Rem}We have shown in the above theorem \ref{ch5:t:dimcx} that $\mathscr N$
	forms a basis of $\com{\X}{}\smallsetminus [\com{\X}{}]_0$.
	We now show that this basis depends on a basis (as vector space) of $\X$.
	
	(i) If $\X$ be a Hausd\"{o}rff topological vector space of dimension $1$, then
	any non-zero element of $\X$ is a scalar multiple of a single basic vector of $\X$
	and hence $\mathscr N$ contains exactly one element. So $\dim\com{\X}{}=[1:1]$.
	For that reason dimension of $\com{\R}{}$ over $\R$ is $[1:1]$ and
	dimension of $\com{\C}{}$ over $\C$ is $[1:1]$. 
	
	(ii) Let $\X$ be a Hausd\"{o}rff topological vector space of dimension $2$ and $B=\{a,b\}$ be a basis of $\X$.
	We first show that $\X'=\{a+\beta b:\beta \in\K\}\cup\{b\}$, where $\X'$ is 
	as defined in the proof of the theorem \ref{ch5:t:dimcx}.
	Any two distinct elements $a+\beta_1b,a+\beta_2 b\in\X\smallsetminus\{\theta\}$
	must lie in two different equivalence classes relative to `$\sim$', since for
	any $\alpha\in \K^*$ if $\alpha (a+\beta_1 b)=a+\beta_2 b$ then $\alpha=1$ 
	and hence $\beta_1=\beta_2$ [as $\{a,b\}$ is a linearly independent subset of $\X$] --- this contradicts that $a+\beta_1b\neq a+\beta_2 b$. Also linear 
	independence of $a,b$ implies that $a+\beta b$ and $b$ must lie in two 
	different equivalence classes relative to `$\sim$', for any $\beta\in\K$.
	Now for any non-zero element $x\in\X$, $\exists\,\alpha,\beta\in\K$ (not both
	zero) such that $x=\alpha a+\beta b$ [since $\{a,b\}$ is a basis of $\X$]. If
	$\alpha \neq 0$ then $x=\alpha(a+\beta \alpha ^{-1} b)$ $\Rightarrow$
	$x$ lies in the class (relative to `$\sim$') whose representative is 
	$(a+\beta \alpha ^{-1} b)$. If $\alpha =0$ then $x$ lies in the equivalence
	class (relative to `$\sim$') whose representative is $b$. Therefore 
	$\X'=\{a+\beta b:\beta \in \K\}\cup\{b\}$. Now the map
	$\alpha \longmapsto a+\alpha b$ creates a bijection between $\K$ and $\X'\smallsetminus \{b\}$.
	So we can say that cardinality of $\X'$ and hence cardinality of 
	$\mathscr N$ is $c$, the cardinality of the set of real numbers $\R$. 
	Therefore $\dim \com{\X}{}=[c:2]$.
	For that reason dimension of $\com{\C}{}$ over $\R$ is $[c:2]$.
	
	(iii) In a similar manner as above we can show that for a  \emph{well-ordered} basis 
	$B$ of a Hausd\"{o}rff topological vector space $\X$,\\ 
	\centerline{$\X'=(e_1+<B_1>)\cup (e_2+<B_2>)\cup\cdots \cup(e_n+<B_n>)\cup\cdots$}
	where $B=\{e_1,e_2,\dots,e_n,\dots\}$, $B_1=B\smallsetminus \{e_1\}$, 
	$B_n=B_{n-1}\smallsetminus \{e_n\}$, $\forall\,n\geq2$ and $<B_i>$ denotes the 
	linear span of $B_i$ in $\X$, $\forall\,i$.\end{Rem}

\begin{Th} For every vector space $\X$, the evs $\mathscr L(\X)$ has a basis. {\em [The evs $ \mathscr L(\X) $ is discussed in Example \ref{ch3:e:LX}]}
	\label{ch:e:lx}\end{Th}

\begin{proof} Let $\mathscr T$ be the collection of all one dimensional subspaces
	of $\X$. We now show that $\mathscr T$ forms a basis of $\mathscr L(\X)
	\smallsetminus[\mathscr L(\X)]_0$. For
	any non-trivial subspace $\mathcal Y$ of $\X$,
	there exists a non-zero element $x\in\mathcal{Y}$ such that $<x>\,\subseteq\mathcal Y$ \big[here
	$<x>$ denotes the linear span of $x$ in $\X$\big]. So $\mathcal Y\in L(<x>)$. Also
	$<x>\,\in \mathscr T$. Thus $\mathscr T$ generates 
	$\mathscr L(\X)\smallsetminus[\mathscr L(\X)]_0$. For any two distinct elements
	$<x>,<y>\,\in \mathscr T$, if $\alpha<x>\,\subseteq\,<y>$ for some $\alpha\in\K^*$
	then $<x>\,=\alpha<x>\,\subseteq\,<y>$ $\Rightarrow$ $<x>\,=\,<y>$ which
	contradicts that $<x>$ and $<y>$ are distinct. So we can say that $\mathscr T$
	is an orderly independent subset of $\mathscr L(\X)\smallsetminus[\mathscr L(\X)]_0$.
	Therefore $\mathscr T$ forms a basis of $\mathscr L(\X)\smallsetminus[\mathscr L(\X)]_0$.
	Consequently $\mathscr L(\X)$ has a basis and $\dim\mathscr L(\X)=[\cd{\mathscr T}:0]$,
	since $[\mathscr L(\X)]_0=\big\{\{\theta\}\big\}$.
\end{proof}

From above theorem we can immediately get the following result. Also $\mathscr T$ is the \emph{only} basis of 
$\mathscr L(\X)\smallsetminus[\mathscr L(\X)]_0$.

\begin{Res}$\dim\mathscr L(\X)=[1:0]$, when $\dim \X=1$ and 
	$\dim\mathscr L(\X)=[c:0]$, when $\dim \X=2$, $c$ being the 
	cardinality of the set of all reals $\R$.\end{Res}
	
\begin{Not}
From the previous result we can say that $\dim\mathscr L(\R)=[1:0]$ which is same with the $\dim [0,\infty)$. But $\mathscr L(\R)$ and $[0,\infty)$ are not order-isomorphic as first one is non-topological evs whereas second one is a topological evs and being topological is an evs property. This example shows that converse part of the statement that equality of dimension is an evs property which we have discussed in  \ref{equality dimension} is not true.
\end{Not}	

\begin{Th}For any $n\in\N$, $\mathscr D^n[0,\infty)$ has a basis and 
	$\dim \mathscr D^n[0,\infty)=[1:0]$.\end{Th}

\begin{proof} We first show that $(0,0,\dots,0,1)$ generates 
	$\mathscr D^n[0,\infty)\smallsetminus [\mathscr D^n[0,\infty)]_0$.
	Let\\ $x=(x_1,\dots,x_n)\in\mathscr D^n[0,\infty)\smallsetminus [\mathscr D^n[0,\infty)]_0$.
	Since $[\mathscr D^n[0,\infty)]_0=\big\{(0,\dots,0)\big\}$, there exists 
	$i\in \{1,2,\dots,n\}$ such that $x_i\neq 0$ and
	$x_j=0$, for all $j<i$. If $i<n$ then obviously $(0,0,\dots,0,1)\leq x$. If 
	$i=n$ then $\frac{x_i}{2}(0,0,\dots,0,1)\leq x$.
	In any case $x\in L\big((0,0,\dots,1)\big)$. Since $\big\{(0,\dots,0,1)\big\}$
	is orderly independent it follows that $\big\{(0,\dots,0,1)\big\}$ is a basis 
	of $\mathscr D^n[0,\infty)\smallsetminus[\mathscr D^n[0,\infty)]_0$ and 
	hence $\dim\mathscr D^n[0,\infty)=[1:0]$.
\end{proof}

Following exampale shows that there exist an evs which has no basis. 

\begin{Th} $X:=\mathscr D\big([0,\infty):\N\big)$ has no basis.\end{Th}

\begin{proof} Let $x=(x_i)_{i\in\N}\in X\smallsetminus X_0$. Since here
	$X_0=\big\{(0,0,\dots)\big\}$, there must exist a least positive integer
	$p$ such that $x_p\neq 0$. If we consider 
	$y=(y_i)_{i\in \N}$, where $y_i=x_i$, $\forall\,i\neq p,p+1$ and $y_p=0$, $y_{p+1}=1$
	then $y\leq x$ and $y\notin X_0$; but there does not exist any $\alpha \in \K^*$ 
	such that $\alpha x\leq y$ ------ which means that $y\notin L(x)$. This shows 
	that $x\notin Q(X)$ and this holds for any non-zero element $x$ of $X$. Therefore
	$Q(X)=\emptyset$. So $\mathscr D\big([0,\infty):\N\big)$ has
	no basis.\end{proof}

Looking at the proof of the above theorems we can get the following 
generalised theorem.

\begin{Th}For a well-ordered set $I$, $\mathscr D(X:I)$ has
	a basis iff $I$ has a maximum element.\end{Th}

\end{document}